\documentclass{amsart}                
\usepackage{amsmath,amsthm}
\usepackage{amsfonts,anysize}
\usepackage{amssymb}
\usepackage[misc]{ifsym}
\usepackage{enumerate, booktabs, color}
\newcommand\numberthis{\addtocounter{equation}{1}\tag{\theequation}}
\newcommand{\lbar}{\bar{\ell}}
\newcommand{\herm}{\mathrm{H}}
\newcommand{\symp}{\mathrm{W}}
\newcommand{\PG}{\mathrm{PG}}
\newcommand{\GQ}{\Gamma}
\newcommand{\subGQ}{\Gamma'}

\newtheorem{theorem}{Theorem}[section]
\newtheorem{lemma}[theorem]{Lemma}
\newtheorem{proposition}[theorem]{Proposition}
\newtheorem{corollary}[theorem]{Corollary}

\begin{document}

\title[A relative $m$-cover of a Hermitian surface is a relative hemisystem]{A relative $m$-cover of a Hermitian surface\\ is a relative hemisystem}

\author{John Bamberg \and Melissa Lee}

\address{
Centre for the Mathematics of Symmetry and Computation,\\ 
School of Mathematics and Statistics\\
The University of Western Australia,\\
35 Stirling Highway, Crawley, WA 6009, Australia. }

 \email{john.bamberg@uwa.edu.au}   
 \email{melissa.lee@research.uwa.edu.au} 


\maketitle

\begin{abstract}
An \emph{$m$-cover} of the Hermitian surface $\herm(3,q^2)$ of $\PG(3,q^2)$ is a set $\mathcal{S}$ of lines of  
$\herm(3,q^2)$ such that every point of $\herm(3,q^2)$ lies on exactly $m$ lines of $\mathcal{S}$, and $0<m<q+1$.
Segre (1965) proved that if $q$ is odd, then $m=(q+1)/2$, and called such a set $\mathcal{S}$ of lines a \emph{hemisystem}.
Penttila and Williford (2011) introduced the notion of a \emph{relative hemisystem}: a set of lines $\mathcal{R}$ of $\herm(3,q^2)$, $q$ even, disjoint from
a symplectic subgeometry $\symp(3,q)$ such that every point of $\herm(3,q^2)\setminus \symp(3,q)$ lies on exactly $q/2$ elements of $\mathcal{R}$.
In this paper, we provide an analogue of Segre's result by introducing \emph{relative $m$-covers} of $\herm(3,q^2)$ with respect to
a symplectic subgeometry and proving that $m$ must necessarily be $q/2$.
\end{abstract}

\section{Introduction}
\label{intro}
An \emph{$m$-cover} of the Hermitian surface $\herm(3,q^2)$ is a set $\mathcal{S}$ of lines of $\herm(3,q^2)$ such that each point of $\herm(3,q^2)$ lies on precisely $m$ elements of $\mathcal{S}$. An $m$-cover is sometimes also called a \emph{regular system of order $m$}. It was shown by Beniamino Segre (1965) \cite{Segre:1965aa}
that $\herm(3,q^2)$ does not have spreads, that is $1$-covers, for every prime power $q$.
Moreover, Segre determined that the only (nontrivial) $m$-covers of $\mathrm{H}(3,q^2)$, $q$ odd, have $m=\tfrac{q+1}{2}$ \cite{Segre:1965aa}. He called these $m$-covers \textit{hemisystems} 
as they constitute half of the lines on every point. Hemisystems of $\herm(3,q^2)$ have piqued the interest of finite geometers and algebraic graph theorists alike due to their links to 
extremal configurations in combinatorics, such as strongly regular graphs \cite{Cameron:1978aa}, partial quadrangles \cite{Cameron:1975aa} and association schemes \cite{Dam:2013aa}. 

In 2011, Penttila and Williford defined the concept of a \textit{relative hemisystem} on $\herm(3,q^2)$, for $q$ a power of 2 \cite{penttila2011new}. In this context, we consider points and lines that are disjoint or \textit{external} to a symplectic subgeometry $\symp(3,q)$. In particular, a relative hemisystem of $\herm(3,q^2)$  is a set $\mathcal{R}$ of external lines with respect to $\symp(3,q)$ such that each external point lies on $\tfrac{q}{2}$ elements
of $\mathcal{R}$. The interest in relative hemisystems stems from the fact that they give rise to primitive $Q$-polynomial association schemes that cannot be constructed from distance regular graphs (see \cite[Theorem 3]{penttila2011new}). These particular sorts of association schemes were considered rare before the introduction of relative hemisystems, and the first infinite family of relative hemisystems gave rise to the first infinite family of such association schemes. 

In this paper, we generalise the concept of a relative hemisystem naturally to a \textit{relative $m$-cover}. 
Let $\GQ$ be a generalised quadrangle with a subquadrangle $\subGQ$.
A relative $m$-cover of $\GQ$ is a set $\mathcal{R}$ of external lines with respect to $\subGQ$ such that each external point lies on $m$ elements
of $\mathcal{R}$. By considering relative $m$-covers, we are able to prove an analogue of Segre's result that initiated the study of hemisystems. 
In particular, we prove that every nontrivial relative $m$-cover of a certain generalised quadrangle $\Gamma$ with respect to
a \emph{doubly subtended} subquadrangle $\subGQ$ is a relative hemisystem; which
includes the case that $\Gamma$ is $\herm(3,q^2)$ and $\subGQ$ is $\symp(3,q)$. The definition of a doubly subtended subquadrangle
appears in the next section in the context of subtended spreads of lines. 

\begin{theorem}
\label{bigthm}
Let $\GQ$ be a generalised quadrangle of order $(q^2,q)$ and suppose it has a doubly subtended subquadrangle $\subGQ$ 
of order $(q,q)$. Let $\mathcal{R}$ is a nontrivial relative $m$-cover of $\GQ$ with respect to $\subGQ$. Then:
\begin{enumerate}[(a)]
\item $q$ is even and $\mathcal{R}$ is a relative hemisystem, that is, $m=\frac{q}{2}$.
\item If $\sigma$ is an involutory automorphism fixing $\subGQ$ point-wise, then the image of $\mathcal{R}$ under $\sigma$ is its complement.
\end{enumerate} 
\end{theorem}

Thas \cite{Thas:1998aa} showed that when $q$ is even, the only generalised quadrangle $\Gamma$ of order $(q^2,q)$ with a doubly subtended subquadrangle $\Gamma'$ of order $(q,q)$ is when $\Gamma=\herm(3,q^2)$ and 
$\Gamma'=\symp(3,q)$. For $q$ odd, the only known examples are
$\herm(3,q^2)$ (with doubly subtended quadrangle $\symp(3,q)$) and the \emph{Kantor-Knuth} flock generalised quadrangles.
(For \emph{dual translation generalised quadrangles}, K. Thas \cite{Thas:2007aa} showed that these two families of examples
are precisely the only generalised quadrangles of order $(q^2,q)$ with doubly subtended subquadrangles).

\newpage

\begin{corollary}\leavevmode
\begin{enumerate}
\item[\textrm{(a)}] A nontrivial relative $m$-cover of $\herm(3,q^2)$ with respect to a symplectic subgeometry $\symp(3,q)$
is a relative hemisystem.
\item[\textrm{(b)}] A Kantor-Knuth flock generalised quadrangle of order $(q^2,q)$, $q$ odd, does not have
a nontrivial relative $m$-cover with respect to a double subtended subquadrangle.
\end{enumerate}
\end{corollary}

We leave as an open problem whether or not the first part of Theorem \ref{bigthm} holds generally for relative $m$-covers of nonclassical generalised quadrangles of order $(q^2,q)$.

%
%

\section{Preliminaries}
\label{prelims}

The central object of study in this paper is the Hermitian surface $\herm(3,q^2)$, whose points and lines form a generalised quadrangle, and the language and theory of
generalised quadrangles will be beneficial. Furthermore, some of our preliminary results on relative $m$-covers hold in the more general context of generalised quadrangles,
which may be of interest. A \textit{generalised quadrangle of order $(s,t)$} is a point-line incidence structure satisfying the following axioms.
\begin{itemize}
\item Any two points are incident with at most one line.
\item Every point is incident with $t+1$ lines.
\item Every line is incident with $s+1$ points.
\item (GQ-axiom) For any point $P$ and line $\ell$ that are not incident, there is a unique point $P'$ on $\ell$ that is collinear with $P$.
\end{itemize}
Given a generalised quadrangle $\GQ$, we say that $\subGQ$ is a \emph{subquadrangle} of $\GQ$ if it is a generalised quadrangle, and if the sets of points and lines of $\subGQ$ are proper 
subsets of the points and lines of $\GQ$. The \textit{dual} of a generalised quadrangle of order $(s,t)$ is obtained by swapping the point and line sets while maintaining incidence. The dual is also a generalised quadrangle, and it is of order $(t,s)$.
For additional background on generalised quadrangles, see \cite{Payne:2009aa}.
In this paper we are interested in generalised quadrangles $\Gamma$ of order $(q^2,q)$, for some $q$ a power of two, and their subquadrangles $\subGQ$ of order $(q,q)$. The classical generalised quadrangles of order $(q^2,q)$ are the Hermitian spaces $\herm(3,q^2)$, obtained by taking the totally isotropic subspaces of $\PG(3,q^2)$ with respect to a Hermitian form. The symplectic space $\symp(3,q)$, which is obtained by taking the points of $\PG(3,q)$ together with the totally isotropic lines with respect to an alternating form,  is a generalised quadrangle of order $(q,q)$ and can be embedded as a subgeometry (and indeed a \textit{subquadrangle}) of $\herm(3,q^2)$
(see \cite[\S 3.5(a)]{Payne:2009aa}). We call points and lines that lie in $\herm(3,q^2)$ but not in $\symp(3,q)$ \textit{external points} and \textit{external lines} respectively. We denote the set of external points by $\mathcal{P}_E$ and the set of external lines by $\mathcal{L}_E$. A simple calculation shows that 
\begin{equation}\label{basiccount}
|\mathcal{P}_E| = (q^2+1)(q^3-q)\quad \text{and} \quad |\mathcal{L}_E| = q^2(q^2-1).
\end{equation}
We call a subset $\mathcal{R}$ of external lines a \textit{relative $m$-cover} of $\GQ$ with respect to $\subGQ$ if every external point is incident with exactly $m$ lines of $\mathcal{R}$. 
There is always a relative $0$-cover and a relative $q$-cover of $\GQ$ with respect to $\subGQ$, obtained by taking none or all of the external lines. We will call these two cases \textit{trivial}.
 
 \begin{lemma}
 \label{RmCsize}
 Suppose that $\mathcal{R}$ is a relative $m$-cover of $\GQ$ with respect to a subquadrangle $\subGQ$. Then $|\mathcal{R}| = m(q^3-q)$.
 \end{lemma}
 
 \begin{proof}
We double count pairs $(X,\ell)$, where $X$ is an external point and $\ell$ is a line of the relative $m$-cover incident with $X$.
By Equation \eqref{basiccount}, we have $(q^2+1)(q^3-q)\cdot m = |\mathcal{R}|(q^2+1)$, and so $|\mathcal{R}| = m(q^3-q)$.
\end{proof}

 A \textit{spread} of a generalised quadrangle is a set $S$ of lines such that each point is incident with precisely one element of $S$. Given an external line $\ell$, it can be seen that the lines concurrent with $\ell$ and meeting a subquadrangle $\subGQ$ in $q+1$ points form a spread $\mathcal{S}_\ell$ of $\subGQ$. We say that $\mathcal{S}_\ell$ is the spread \textit{subtended} by $\ell$. Furthermore, a spread of $\subGQ$ is \textit{doubly subtended} if it is subtended by two external lines $\ell$ and $\bar{\ell}$, and we then say that $\ell$ and $\bar{\ell}$ are \textit{antipodes}. We say that $\subGQ$ is \textit{doubly subtended} if every 
 subtended spread of $\subGQ$ is doubly subtended.  
Generalised quadrangles of order $(q^2,q)$ with doubly subtended subquadrangles of order $(q,q)$ were first considered by Brown \cite{brown1997generalized}, albeit in the context of the dual. 
Brown \cite[Corollary 2.2]{brown1997generalized} showed that the size of the intersection of two subtended spreads $\mathcal{S}_\ell$ and $\mathcal{S}_n$ is one of the following:
\begin{center}
\begin{tabular}{ll}
\toprule
$q^2+1$ & if $\mathcal{S}_\ell=\mathcal{S}_n$,\\
$1$ & if $\mathcal{S}_\ell\ne \mathcal{S}_n$ and $n$ is concurrent with either $\ell$ or $\bar{\ell}$,\\
$q+1$ & if $\mathcal{S}_\ell\ne \mathcal{S}_n$ and $n$ is not concurrent with $\ell$ nor $\bar{\ell}$.\\
\bottomrule
\end{tabular}
\end{center}
 He further showed that a generalised quadrangle of order $(q^2,q)$ has a doubly subtended subquadrangle if and only if there is an involutory automorphism that fixes the subquadrangle pointwise. This involutory automorphism swaps antipodes $\ell$ and $\bar{\ell}$.

\begin{lemma}
\label{barlemma}
 Let $\ell$ be an external line. Then a line meeting $\ell$ is concurrent with $\bar{\ell}$ if and only if it also meets $\subGQ$.
\end{lemma}
\begin{proof}
By definition, $\ell$ and $\lbar$ subtend the same spread of $\subGQ$, which means that every line concurrent with $\ell$ that meets $\subGQ$ must also be concurrent with $\lbar$. As for the `if' implication, suppose $k$ is a line concurrent with both $\ell$ and $\lbar$. Then  the point that is the intersection of $\ell$ and $k$ must have one line on it meeting $\subGQ$, and, by the previous argument, that line must be concurrent with $\lbar$. To prevent the existence of a triangle, $k$ must meet $\subGQ$.
\end{proof}

\begin{lemma}
\label{extlinesint}
Let $\ell$ and $n$ be two nonconcurrent external lines in a generalised quadrangle $\GQ$ of order $(q^2,q)$ containing a doubly subtended subquadrangle $\subGQ$ of order $(q,q)$. Then the number of external lines concurrent with both $\ell$ and $n$ is:
\begin{enumerate}[(a)]
\item $0$ if $n = \bar{\ell}$,
\item $q-2$ if $n$ is concurrent with $\ell$,
\item $q^2-q$ if $n$ is not concurrent with either $\ell$ or $\bar{\ell}$, or
\item $q^2$ if $n$ is concurrent with $\bar{\ell}$.
\end{enumerate}
\end{lemma}
\begin{proof}
Suppose $n$ is concurrent with $\ell$ and that $P$ is their point of intersection. To prevent the existence of a triangle, the only lines that can be concurrent with both $\ell$ and $n$ (but not equal to either) are the $q-2$ remaining external lines that are incident with $P$.  Now suppose $n$ is not concurrent with or equal to $\ell$. By the `GQ-axiom', $\ell$ and $n$ are concurrent with $q^2+1$ lines together in $\GQ$. We must determine how many of these are external lines. The lines concurrent with both $\ell$ and $n$ and meeting $\subGQ$ are exactly the lines in the intersection of the spreads subtended by $\ell$ and $n$. Therefore, if $\ell$ and $n$ subtend the same spread (that is, $n= \bar{\ell}$), then there must be zero external lines concurrent with both. If the intersection of the spreads subtended by $\ell$ and $n$ has size $q+1$ (that is, $n$ is concurrent with neither $\ell$ or $\bar{\ell}$), then $\ell$ and $n$ are concurrent with $q^2-q$ external lines together. Finally, if $n$ is concurrent with $\bar{\ell}$, which means that size of the intersection of their subtended spreads is equal to one, then $\ell$ and $n$ must be concurrent with $q^2$ external lines together.
\end{proof}

For an overview of the theory of association schemes, see \cite{Lint:2001aa}.
Penttila and Williford \cite{penttila2011new} showed that there is an association scheme on the external points of a generalised quadrangle of order $(q,q^2)$ with a doubly subtended subquadrangle of order $(q,q)$ which arises from considering the intersection of subtended ovoids. Since we are interested in $\herm(3,q^2)$, we state the association scheme in terms of the dual generalised quadrangle, which means we now have an association scheme on external lines. The original statement of the following result stipulated that $q$ is a power of $2$. However, it can be
readily deduced from its proof that the result does not need this requirement.
\begin{theorem}[Penttila and Williford {\cite[Theorem 1]{penttila2011new}}]
\label{extlinesAS}
Let $\GQ$ be a generalised quadrangle of order $(q^2,q)$ containing a doubly subtended generalised quadrangle $\subGQ$ of order $(q,q)$. Let $\mathcal{L}_E$ be the set of external lines of $\GQ$ relative to $\subGQ$. Then the following set of relations on $\mathcal{L}_E$, together with the equality relation $\Lambda_0$, form a cometric association scheme on $\mathcal{L}_E$.
\begin{center}
\begin{tabular}{ll}
\toprule
\textsc{Relation} & \textsc{Description}\\
\midrule
$(\ell,n)\in \Lambda_1$ &$\ell$ and $n$ are not concurrent and $n$ is concurrent with $\bar{\ell}$.\\
$(\ell,n)\in \Lambda_2$ & $\ell$ and $n$ are not concurrent and $n$ is not concurrent with $\bar{\ell}$.\\
$(\ell,n)\in \Lambda_3$ & $\ell$ and $n$ are concurrent.\\
 $(\ell,n)\in \Lambda_4$ & $n = \bar{\ell}$. \\
\bottomrule
\end{tabular}
\end{center}

\end{theorem}
The unique basis of minimal idempotents $\{E_i \mid 0 \leqslant i \leqslant 4\}$ of the associated Bose--Mesner algebra $\mathcal{A}$ can be constructed once we know the
dual eigenmatrix $Q$ of the association scheme which we give below (see also \cite[p.~505]{penttila2011new}).
Each of these minimal idempotents projects onto a simultaneous eigenspace of $\mathcal{A}$, and together these give a decomposition 
$\mathbb{C}^{\mathcal{L}_E} = V_0 \perp \dots \perp V_4$ of the vector space over $\mathbb{C}$ with the set of external lines $\mathcal{L}_E$ as a basis. Note that $V_0$ is generated by the all-ones vector $j$.

\begin{equation}
\label{Q}
Q = 
\begin{bmatrix}
1 & \frac{q(q-1)^2}{2} & \frac{(q-2)(q+1)(q^2+1)}{2} & \frac{q(q-1)(q^2+1)}{2} & \frac{q(q^2+1)}{2} \\[1em]
1 & \frac{q(q-1)}{2}& \frac{(q-2)(q+1)}{2} & \frac{-q(q-1)}{2} & \frac{-q(q-1)}{2}\\[1em]
1 & 0 & -(q+1) & 0 & q \\[1em]
1 & \frac{-q(q-1)}{2}& \frac{(q-2)(q+1)}{2} & \frac{q(q-1)}{2} & \frac{-q(q-1)}{2}\\[1em]
1 & \frac{-q(q-1)^2}{2} & \frac{(q-2)(q+1)(q^2+1)}{2} & \frac{-q(q-1)(q^2+1)}{2} & \frac{q(q^2+1)}{2} \\
\end{bmatrix}.
\end{equation}

\section{Proof of Main Theorem}
Let $\Gamma= \herm(3,q^2)$, for some prime-power $q$, and let $\subGQ = \symp(3,q)$.
Suppose $\mathcal{S}$ is a set of lines of $\GQ$. Define $\mathcal{S}^{\perp_E}$ to be the set of external lines that are concurrent with every member of $\mathcal{S}$.
We denote the characteristic vector of $\mathcal{S}$ in $\mathcal{L}_E$ by $\chi_\mathcal{S}$.
Let $[P]$ denote the set of lines incident with a particular point $P$, with $\chi_{[P]}$ being its characteristic vector. If $\mathcal{R}$ is a relative $m$-cover 
then the scalar product $\chi_{[P]} \cdot \chi_{\mathcal{R}}$ is equal to $m$, by the definition of a relative $m$-cover, and by Lemma \ref{RmCsize}.
Therefore, we have:
\[
 \big( (q^3-q) \chi_{[P]} - j\big) \cdot \chi_{\mathcal{R}} 
 = m(q^3-q) - m(q^3-q)
 = 0.
\]
Also, recalling the association scheme on external lines given in Theorem \ref{extlinesAS}, and the associated adjacency and minimal idempotent matrices, we have
\begin{align*}
 \left( (q^3-q) \chi_{[P]} - j\right) E_0 & =  \frac{1}{q^2(q^2-1)}\left( (q^3-q) \chi_{[P]} - j\right) J \\
& = \frac{(q^3-q)}{q^2(q^2-1)}\ (qj)  - \frac{q^2(q^2-1)}{q^2(q^2-1)}j\\
& = 0.
\end{align*}
By \eqref{Q}, we can express each of the other $E_i$ matrices in terms of adjacency matrices as $E_i=\frac{1}{|\mathcal{L}_E|}\sum_{j=0}^{4}Q_{ji}A_j$, where $Q_{ji}$ is the $(j,i)$-entry of $Q$, or in more detail:
{\small
\begin{align*}
\numberthis\label{E1}
E_1 &= \frac{1}{|\mathcal{L}_E|} \left(\frac{q(q-1)^2}{2} \mathrm{I} + \frac{q(q-1)}{2} A_1  - 
\frac{q(q-1)}{2} A_3 - \frac{q(q-1)^2}{2} A_4\right), \\
E_2 &= \frac{1}{|\mathcal{L}_E|} \left(\frac{(q-2)(q+1)(q^2+1)}{2} \mathrm{I} + \frac{(q-2)(q+1)}{2} A_1  - (q+1)A_2 \right.\\
&\qquad \left. +\, \frac{(q-2)(q+1)}{2} A_3 + \frac{(q-2)(q+1)(q^2+1)}{2} A_4 \right),\numberthis\label{E2} \\
E_3 &= \frac{1}{|\mathcal{L}_E|} \left(\frac{q(q-1)(q^2+1)}{2} \mathrm{I} - \frac{q(q-1)}{2} A_1 +
\frac{q(q-1)}{2} A_3 - \frac{q(q-1)(q^2+1)}{2} A_4\right),\numberthis\label{E3} \\
E_4&= \frac{1}{|\mathcal{L}_E|} \left(\frac{q(q^2+1)}{2} \mathrm{I} - \frac{q(q-1)}{2} A_1 + qA_2-
\frac{q(q-1)}{2} A_3 + \frac{q(q^2+1)}{2} A_4\right). \numberthis\label{E4} 
\end{align*}}

Let $\Gamma$ be a generalised quadrangle of order $(q^2,q)$ with a doubly subtended subquadrangle $\Gamma'$ of order $(q,q)$.
Let $\sigma$ be the involutory automorphism of $\Gamma$ fixing $\Gamma'$ pointwise.
The following result shows that the \emph{dual degree set} of $\chi_{[P]}$, where $P$ is an external point,
is $\{1\}$.

\begin{proposition}
\label{chipinspan}
Let $P$ be an external point. Then $\chi_{[P]} E_1 = 0$ and $\chi_{[P]} E_i \neq 0$ for $i \in \{2,3,4\}$.
Hence, $\chi_{[P]}  \in V_0\perp V_2\perp V_3\perp V_4$.
\end{proposition}
 \begin{proof}
Let $\bar{P} = P^\sigma$ and define $W$ be the set of external lines concurrent with the line on $P$ meeting the doubly subtended quadrangle $\subGQ$, but not incident with $P$ or $\bar{P}$. Since $\sigma$ is an automorphism,  $P$ is incident with  $\ell $ if and only if $ \bar{P}$ is incident with  $\bar{\ell}$.
 Now, $\chi_{[{P}]} \mathrm{I} = \chi_{[{P}]}$ trivially, and $\chi_{[{P}]} A_4 = \chi_{[\bar{P}]}$ because the $i$th position in the resulting vector will be 1 if and only if the image of the corresponding external line under $\sigma$ is incident with $P$.
Let us now calculate $ \chi_{[P]} A_3$. This is equivalent to counting how many external lines on $P$ are concurrent with each of the other external lines. If $\ell$ is an external line incident with $P$, then the other $q-1$ external lines on $P$ are concurrent with it and so the corresponding positions in $ \chi_{[P]} A_3$ will have value $q-1$. If $\ell$ is an external line not incident with $P$, then, by the `GQ-axiom', either $\ell$ is concurrent with an external line on $P$, in which case the corresponding value in $ \chi_{[P]} A_3$ will be 1, or $\ell$ is concurrent with the line on $P$ meeting the subquadrangle $\subGQ$, and so contributes 0 to $ \chi_{[P]} A_3$. Summarising, we have
$$\chi_{[P]} A_3 = (q-1)\chi_{[P]} + (j-\chi_{[P]} - \chi_W-\chi_{[\bar{P}]}) = j + (q-2)\chi_{[P]} - \chi_W-\chi_{[\bar{P}]}.$$
We now calculate $ \chi_{[P]} A_1$. This is equivalent to counting how many external lines on $P$ are concurrent with the image of each of the other external lines under $\sigma$. Since $\sigma$ preserves incidence, if $\ell$ is a line incident with $P$, then an external line $\bar{n}$ is concurrent with $\ell$ if and only if $\bar{\ell}$ is concurrent with $n$. Therefore, we may equivalently count how many lines incident with $\bar{P}$ are concurrent with each external line $n$, which is the same as calculating $\chi_{[\bar{P}]} A_3$. 
Therefore,
 $$\chi_{[P]} A_1 = (q-1)\chi_{[\bar{P}]} + (j-\chi_{[P]}-\chi_W-\chi_{[\bar{P}]}) = (q-2)\chi_{[\bar{P}]} + j -\chi_W - \chi_{[P]}.$$
We will now calculate $\chi_{[P]} A_2$, using the fact that the sum of the adjacency matrices is the all-ones matrix $J$:
\begin{align*}
 \chi_{[P]} A_2 &= \chi_{[P]} J - \left( \chi_{[P]} I + \chi_{[P]} A_1 + \chi_{[P]} A_3 + \chi_{[P]} A_4 \right)\\
& = qj - \Big( \chi_{[P]} + (q-2)\chi_{[\bar{P}]} + j -\chi_W - \chi_{[P]} + j \\ 
& \qquad + (q-2)\chi_{[P]} - \chi_W-\chi_{[\bar{P}]}+ \chi_{[\bar{P}]}  \Big)\\
& = (q-2)j - (q-2)(\chi_{[\bar{P}]} + \chi_{[P]}) + 2 \chi_W.
\end{align*}

With respect to the set $\{j,\chi_{[P]},\chi_W,\chi_{[\bar{P}]}\} $ we can summarise our calculations as follows:

\begin{align*}
\left(\chi_{[P]} A_i \right) &=
(\chi_{[P]},j,\chi_W,\chi_{[\bar{P}]})
\begin{bmatrix}
 1 & -1 & -(q-2) & q-2 & 0 \\
 0 & 1 & q-2 & 1 & 0 \\
 0 & -1 & 2 & -1 & 0 \\
 0 & q-2 & -(q-2) & -1 & 1 \\
\end{bmatrix}.
\end{align*}
Therefore,
{\footnotesize
\begin{align*}
\frac{1}{|\mathcal{L}_E|}( \chi_{[P]} A_i )Q&=\frac{1}{q^2(q^2-1)}(\chi_{[P]},j,\chi_W,\chi_{[\bar{P}]})
\begin{bmatrix}
 1 & -1 & -(q-2) & q-2 & 0 \\
 0 & 1 & q-2 & 1 & 0 \\
 0 & -1 & 2 & -1 & 0 \\
 0 & q-2 & -(q-2) & -1 & 1 \\
\end{bmatrix}Q\\
&=\frac{(\chi_{[P]},j,\chi_W,\chi_{[\bar{P}]})}{q^2(q^2-1)}
\begin{bmatrix}
0 & 0 & \tfrac{1}{2}q(q-2)(q+1)^2 & \tfrac{1}{2}q^2(q^2-1) & q(q+1)\\
q & 0 & 0 & 0 & -q \\
0 & 0 & -q(q+1) & 0 & q(q+1) \\
0 & 0 & \tfrac{1}{2}q(q-2)(q+1)^2 & -\tfrac{1}{2} q^2(q^2-1) & q(q+1)\\
\end{bmatrix}
\end{align*}
}
Note that only the second column of the matrix above is zero, and so $\chi_{[P]} E_1 = 0$ and $\chi_{[P]} E_i \neq 0$ for $i \in \{2,3,4\}$.
\end{proof}

\begin{theorem}
\label{chipsspan}
Let $\chi_{[P]}$ be the characteristic vector of the set of external lines incident with an external point $P$. Then the set of vectors $\{ \chi_{[P]} \mid P \in \mathcal{P}_E\}$ is a spanning set of $V_0\perp V_2 \perp V_3 \perp V_4$.
\end{theorem}
\begin{proof}
By Proposition \ref{chipinspan}, $\chi_{[P]} \in V_0\perp V_2 \perp V_3 \perp V_4 $. Let $A$ be the matrix whose rows are the $\chi_{[P]}$ vectors. To prove that $\{ \chi_{[P]} \mid P \in \mathcal{P}_E\}$ spans  $V_0\perp V_2 \perp V_3 \perp V_4 $, it is sufficient to show that the rank of $A$ is equal to  $\mathrm{dim}(V_0\perp V_2 \perp V_3 \perp V_4) = |\mathcal{L}_E| - \mathrm{dim}(V_1)$. 
Consider the matrix $M=A^TA$. The rows of $M$ correspond to counting how many points incident with a particular external line are incident with each of the other external lines. The diagonal entries of $M$ are each equal to $q^2+1$, the entries whose row and column correspond to concurrent lines are equal to 1, and the entries are zero otherwise. In particular, the rows of $M$ are of the form $(q^2+1)\chi_\ell + \chi_{\{\ell\}^{\perp_E}}$ for some external line $\ell$. Denote the row of $M$ corresponding to the external line $\ell$ by $M_\ell$, and recall that the size of $\{\ell\}^{\perp_E}$ is $(q-1)(q^2+1)$, since each of the points of $\ell$ is incident with $q-1$ external lines not equal to $\ell$. Note that $\mathrm{rank}(M) \leqslant \mathrm{rank}(A)$, since $M$ is constructed by taking linear combinations of the rows of $A$. Also notice that $M$ has constant column sum equal to $q(q^2+1)$, and so the all-ones vector $j$ lies in the row space of $M$. We first show that $\chi_\ell + \chi_{\bar{\ell}}$ is in the row space of $M$. Define the set 
\[
U_{\ell} := \{ M_\ell\} \cup \{ M_{\ell'} \mid \ell' \text{ is concurrent with } \ell \}.
\]
Then the sum of the vectors in $U_\ell$ is 
\begin{align*}
&(q^2+1)\chi_\ell + \chi_{\{\ell\}^{\perp_E}} + (q^2+1)\chi_{\{\ell\}^{\perp_E}} + (q-1)(q^2+1)\chi_\ell \\
& \quad + (q-2)\chi_{\{\ell\}^{\perp_E}} + q^2\chi_{\{\bar{\ell}\}^{\perp_E}} + (q^2-q)\chi_{\{\ell, \bar{\ell} \}^{\perp_2}},
\end{align*}
where $\chi_{\{\ell, \bar{\ell} \}^{\perp_2}} $ is the set of lines that are neither concurrent with nor equal to $\ell$ or $\bar{\ell}$.
The first two terms originate from $M_\ell$. The third term arises because $M_{\ell'}$ contributes $(q^2+1)\chi_{\ell'}$ for every $\ell' \in \{\ell\}^{\perp_E}$. Every line of $\{\ell\}^{\perp_E}$ is concurrent with $\ell$ by definition, and this is how the fourth term arises. The fifth term comes about because every element of $\{\ell\}^{\perp_E}$ is also concurrent with the $q-2$ elements of $\{\ell\}^{\perp_E}$ that lie on the same point of intersection with $\ell$ (and no other elements of $\{\ell\}^{\perp_E}$). The sixth and seventh terms arise because for every external line $n$ not equal to $\ell$ or $\bar{\ell}$ and not concurrent with $\ell$, Lemma \ref{extlinesint} implies that the size of $\left\{ \ell, n \right\}^{\perp_E}$ is $q^2$ if $n$ is concurrent with $\bar{\ell}$, and $q^2-q$ otherwise. Simplifying slightly, we write the sum as
\[
q(q^2+1)\chi_\ell+ (q^2+q)\chi_{\ell^{\perp_E}} + q^2\chi_{\bar{\ell}^{\perp_E}} + (q^2-q)\chi_{\{\ell, \bar{\ell} \}^{\perp_2}} .
\]
Consider the sum of all of the vectors in $U_{\ell} \cup U_{\bar{\ell}}$. After some simplification, we arrive at the expression
\begin{equation}
\label{eq1}
q(q^2+1)(\chi_\ell+ \chi_{\bar{\ell}}) + (2q^2+q)(\chi_{\ell^{\perp_E}}+ \chi_{\bar{\ell}^{\perp_E}}) + (2q^2-2q)\chi_{\{\ell, \bar{\ell} \}^{\perp_2}}.
\end{equation}
Recall that $j$ is in the row space of $M$. We calculate the difference between \eqref{eq1} and $(2q^2-2q)j$ and arrive at the following expression:
\[
q(q^2-2q+3)(\chi_\ell+ \chi_{\bar{\ell}}) + 3q(\chi_{\ell^{\perp_E}}+ \chi_{\bar{\ell}^{\perp_E}}).
\]
We now subtract $3q(M_\ell+M_{\bar{\ell}})$:
\begin{align*}
&q(q^2-2q+3)(\chi_\ell+ \chi_{\bar{\ell}}) + 3q(\chi_{\ell^{\perp_E}}+ \chi_{\bar{\ell}^{\perp_E}}) - 3q(M_\ell+M_{\bar{\ell}}) \\ &\quad =q(q^2-2q+3)(\chi_\ell+ \chi_{\bar{\ell}}) - 3q(q^2+1)(\chi_\ell+ \chi_{\bar{\ell}})\\
& \quad = -2q^2(q+1)(\chi_\ell+ \chi_{\bar{\ell}}).
\end{align*}
Therefore, $\chi_\ell+ \chi_{\bar{\ell}}$ is in the row space of $M$. The $\chi_\ell$ vectors form a spanning set for the vector space over $\mathbb{C}^{\mathcal{L}_E}$ by definition. Therefore, the set $\{\chi_\ell E_i \mid \ell \in \mathcal{L}_E\}$ forms a basis for $V_i$ for $0 \leqslant i \leqslant 4$. Hence, to show that $V_i$ lies in the row space of $M$, it is sufficient to show that $\chi_\ell E_i$ does, for any choice of external line $\ell$. We first calculate the $\chi_\ell E_i$, based on their expressions as the linear combinations of adjacency matrices above (see Equations \eqref{E1}, \eqref{E2}, \eqref{E3}, \eqref{E4}):
\begin{align*}
\chi_\ell E_0 &= \frac{1}{q^2(q^2-1)}j, \\
\chi_\ell E_1 &= \frac{1}{2q(q+1)}((q-1)(\chi_\ell - \chi_{\bar{\ell}}) +\chi_{\bar{\ell}^{\perp_E}} - \chi_{\ell^{\perp_E}}),\\
\chi_\ell E_2  &= \frac{1}{2q^2(q-1)}(-2j+ q((q-1)^2\chi_\ell + \chi_{\ell^{\perp_E}} + (q-1)^2\chi_{\bar{\ell}} + \chi_{\bar{\ell}^{\perp_E}})),\\
\chi_\ell E_3 &= \frac{1}{2q(q+1)}((q^2+1)\chi_\ell - (q^2+1) \chi_{\bar{\ell}} -\chi_{\bar{\ell}^{\perp_E}} + \chi_{\ell^{\perp_E}}),\\
\chi_\ell E_4 &= \frac{1}{2q(q^2-1)}(2j+(q+1)((q-1)(\chi_\ell +\chi_{\bar{\ell}}) - \chi_{\bar{\ell}^{\perp_E}} - \chi_{\ell^{\perp_E}})).
\end{align*}
Recall that the rows $M_\ell$ of $M$ are of the form $(q^2+1)\chi_\ell + \chi_{\{\ell\}^{\perp_E}}$ for each external line $\ell$. Then $ \chi_\ell E_3$ is in the row space of $M$, since it is a linear combination of the $M_\ell - M_{\bar{\ell}}$ vectors. Now consider $\chi_\ell E_2$. Since $j$ is in the row space of $M$, we only need to show that $(q-1)^2\chi_\ell + \chi_{\ell^{\perp_E}} + (q-1)^2\chi_{\bar{\ell}} + \chi_{\bar{\ell}^{\perp_E}}$ is also in the row space of $M$. We can rewrite this vector as
\[
((q^2+1)\chi_\ell + \chi_{\ell^{\perp_E}} ) + ((q^2+1)\chi_{\bar{\ell}} + \chi_{\bar{\ell}^{\perp_E}} ) + ((q-1)^2-(q^2+1))(\chi_\ell+\chi_{\bar{\ell}}).
\]
This vector lies in the row space of $M$ since the first two terms are $M_\ell$ and $M_{\bar{\ell}}$, and $\chi_\ell+ \chi_{\bar{\ell}}$ also lies in the row space of $M$. Finally, consider $\chi_\ell E_4$. Using the same reasoning as in the last case, it is sufficient to show that $(q-1)(\chi_\ell +\chi_{\bar{\ell}}) - \chi_{\bar{\ell}^{\perp_E}} - \chi_{\ell^{\perp_E}}$ is in the row space of $M$. We can rewrite this expression as 
\[
-((q^2+1)\chi_\ell + \chi_{\ell^{\perp_E}} ) - ((q^2+1)\chi_{\bar{\ell}} + \chi_{\bar{\ell}^{\perp_E}} ) + ((q-1)+(q^2+1))(\chi_\ell+\chi_{\bar{\ell}}),
\]
which is in the rowspace of $M$. Therefore, $V_0 \perp V_2 \perp V_3 \perp V_4$ is a subspace of the row space of $M$. This implies that $M$, and therefore $A$, has rank at least the dimension of $V_0 \perp V_2 \perp V_3 \perp V_4$. Hence $A$ has rank equal to the dimension of $V_0 \perp V_2 \perp V_3 \perp V_4$ and its rows (that is, the set $\{ \chi_{[P]} \mid P \in \mathcal{P}_E\}$) form a spanning set of $V_0 \perp V_2 \perp V_3 \perp V_4$.
\end{proof}

\begin{corollary}
\label{spancoroll}
The set of vectors $\{(q^3-q)\chi_{[P]} - j \mid P \in \mathcal{P}_E \}$ forms a spanning set of $V_2 \perp V_3 \perp V_4$.
\end{corollary}
\begin{proof}
By Theorem \ref{chipsspan}, the set $\{ \chi_{[P]} \mid P \in \mathcal{P}_E\}$ forms a spanning set of $V_0 \perp V_2 \perp V_3 \perp V_4$.
 Now, $((q^3-q)\chi_{[P]} - j)E_0 = \frac{1}{q^2(q^2-1)}\left((q^3-q)q - q^2(q^2-1)\right) = 0$, so $(q^3-q)\chi_{[P]} - j \in V_2\perp V_3 \perp V_4$ for every external point $P$. Therefore,  the set $\{(q^3-q)\chi_{[P]} - j \mid P \in \mathcal{P}_E \}$ forms a spanning set of $V_2 \perp V_3 \perp V_4$.
\end{proof}

\begin{corollary}
\label{SinV0V1}
Let $\mathcal{R}$ be a relative $m$-cover of a generalised quadrangle $\GQ$ of order $(q^2,q)$ relative to a doubly subtended subquadrangle $\subGQ$ of order $(q,q)$. Then $\chi_\mathcal{R}$ lies in $V_0 \perp V_1$.
\end{corollary}
\begin{proof}
By Corollary \ref{spancoroll}, 
it is sufficient to show that $\chi_\mathcal{R} \cdot ((q^3-q)\chi_{[P]} - j ) = 0$ for all $P \in \mathcal{P}_E$. By Lemma \ref{RmCsize}, we have 
\[
\chi_\mathcal{R} \cdot ((q^3-q)\chi_{[P]} - j ) = m(q^3-q) - m(q^3-q) = 0.
\] 
\end{proof}

We are now ready to prove Theorem \ref{bigthm}.

\begin{proof}[of Theorem \ref{bigthm}]
Let $\mathcal{R}^\sigma$ denote the image of $\mathcal{R}$ under $\sigma$.
We first calculate the product of $\chi_\mathcal{R}$ with each of the adjacency matrices.
 First, $\chi_\mathcal{R}A_4 = \chi_{\mathcal{R}^\sigma}$ since there will be a 1 in the $j$th position of $\chi_\mathcal{R}A_4 $ if and only if the image of the corresponding line under $\sigma$ is contained in $\mathcal{R}$, and 0 otherwise. Let us now calculate $\chi_\mathcal{R}A_3$. This is equivalent to counting how many elements of $\mathcal{R}$ are concurrent with each external line. If $\ell$ is an external line in $\mathcal{R}$, then it will be concurrent with $(m-1)(q^2+1)$ lines of $\mathcal{R}$ (that is, $m-1$ for each point on the line). If $\ell$ is not in $\mathcal{R}$, then there will be $m(q^2+1)$ lines of $\mathcal{R}$ concurrent with $\ell$. Summarising,
\[
 \chi_\mathcal{R}A_3 = (q^2+1)(m-1)\chi_\mathcal{R} + m(q^2+1)(j-\chi_\mathcal{R}) = m(q^2+1)j-(q^2+1)\chi_\mathcal{R}.
\]
We now calculate $\chi_\mathcal{R}A_1$. This is equivalent to counting the number of elements of $\mathcal{R}$ that are concurrent with the image $\bar{\ell}$ of each external line $\ell$ under $\sigma$. Since $\sigma$ preserves incidence, if $n$ is a line in $\mathcal{R}$, then $n$ is concurrent with $\bar{\ell}$ if and only if $\bar{n}$ is concurrent with $\ell$. Therefore, we may equivalently count how many lines in $\mathcal{R}^\sigma$ are concurrent with each external line $\ell$, which is the same as calculating $\chi_{\mathcal{R}^\sigma} A_3$. Using the calculation we just completed, we have 
\[
 \chi_\mathcal{R} A_1 = \chi_{\mathcal{R}^\sigma} A_3 = m(q^2+1)j-(q^2+1)\chi_{\mathcal{R}^\sigma}.
\]
Now, it remains to calculate $\chi_\mathcal{R}A_2$. We will make use of the fact that the sum of all of the adjacency matrices is equal to $J$. Therefore,
\begin{align*}
 \chi_\mathcal{R}A_2 & = \chi_\mathcal{R}J - \left( \chi_\mathcal{R}I + \chi_\mathcal{R} A_1 + \chi_\mathcal{R} A_3 + \chi_\mathcal{R} A_4 \right)\\
& = m(q^3-q)j- \big(\chi_\mathcal{R} +  m(q^2+1)j-(q^2+1)\chi_{\mathcal{R}^\sigma} \\ 
& \qquad + m(q^2+1)j-(q^2+1)\chi_\mathcal{R} + \chi_{\mathcal{R}^\sigma}\big)\\
& = m(q^3-2q^2-q-2)j+q^2(\chi_\mathcal{R} + \chi_{\mathcal{R}^\sigma}).
\end{align*}
We now make use of Equations \eqref{E2}, \eqref{E3} and \eqref{E4} to compute the projections of $\chi_\mathcal{R}$ into $V_2$, $V_3$ and $V_4$. They are as follows:
 \begin{align*}
  \chi_\mathcal{R} E_2 & = \frac{1}{q^2(q^2-1)} \left(2mq(q+1)j - q^2(q+1)(\chi_\mathcal{R} + \chi_{\mathcal{R}^\sigma})\right),\\
  \chi_\mathcal{R} E_3 &= 0, \\
  \chi_\mathcal{R}E_4 &= \frac{1}{q^2(q^2-1)}\left(\frac{q^2(q+1)^2}{2}(\chi_\mathcal{R} + \chi_{\mathcal{R}^\sigma}) -mq(q+1)^2j \right).
 \end{align*}
 By Corollary \ref{SinV0V1}, $\chi_\mathcal{R} \in V_0 \perp V_1$ and therefore $\chi_\mathcal{R} E_i=0$ for $i \in \{2,3,4\}$. 
 We only need to consider $i=2$: here we see that $2m j = q(\chi_\mathcal{R} + \chi_{\mathcal{R}^\sigma})$.
 Therefore, $\chi_\mathcal{R} + \chi_{\mathcal{R}^\sigma}$ is the constant vector with value $2m/q$, and since
 the values of $\chi_\mathcal{R} + \chi_{\mathcal{R}^\sigma}$ lie in $\{0,1,2\}$, it follows immediately that $q$ is even,
 $m=q/2$, and $\chi_\mathcal{R} + \chi_{\mathcal{R}^\sigma}=j$. That is, $\mathcal{R}$ is a relative hemisystem
 and $\mathcal{R}^\sigma$ is the complement of $\mathcal{R}$.
 
\end{proof}

\subsection*{Acknowledgements}
The authors are indebted to Michael Giudici for many discussions on the material in this work.
The first author acknowledges the support of the Australian Research Council Future Fellowship
FT120100036. The second author acknowledges the support of a Hackett Postgraduate Scholarship. The authors also thank Daniel Horsley for his mention of an elementary result during his talk at Combinatorics 2016, which was instrumental in the proof of the main result of this paper.

%

\end{document}